\documentclass{amsart}
\usepackage{amsthm,amsfonts,amssymb,euscript,color}
\usepackage{bbm}
\usepackage{graphics}
\usepackage{enumerate}

\newtheorem{theorem}{Theorem}[section]
\newtheorem{proposition}{Proposition}[section]
\newtheorem{lemma}[proposition]{Lemma}
\newtheorem{remark}{Remark}[section]
\newtheorem{definition}{Definition}[section]

\newcommand{\eqdef}{\overset{\mbox{\tiny{def}}}{=}}

\newcommand{\pel}{p}
\newcommand{\PL}{V}
\newcommand{\pZ}{\pel_0}
\def\vh {\hat{\pel}}
\def\sing {(1+\hat{\pel}\cdot\omega)}

\def\ls {\lesssim}
\def\om {\omega}
\def\th {\theta}
\def\rd {\partial}
\def\Rt {\mathbb R^3}

\newcommand{\bea}{\begin{eqnarray}}
\newcommand{\eea}{\end{eqnarray}}
\def\beaa{\begin{eqnarray*}}
\def\eeaa{\end{eqnarray*}}

\title[Continuation criterion for the relativistic Vlasov-Maxwell system]{A new continuation criterion for the relativistic Vlasov-Maxwell system}

\author{Jonathan Luk}
\address{Department of Mathematics, University of Pennsylvania, Philadelphia, PA 19104}
\email{jwluk@sas.upenn.edu}
\thanks{J.L. was partially supported by the NSF Postdoctoral Fellowship DMS-1204493.}

\author{Robert Strain}
\address{Department of Mathematics, University of Pennsylvania, Philadelphia, PA 19104}
\email{strain@math.upenn.edu}
\thanks{R.M.S. was partially supported by the NSF grant DMS-1200747, and an Alfred P. Sloan Foundation Research Fellowship.}

\begin{document}

\begin{abstract}
The global existence of solutions to the relativistic Vlasov-Maxwell system given sufficiently regular finite energy initial data is a longstanding open problem.  The main result of Glassey-Strauss \cite{GS86} shows that a solution $(f, E, B)$ remains $C^1$ as long as the momentum support of $f$ remains bounded.  Alternate proofs were later given by Bouchut-Golse-Pallard \cite{BGP} and Klainerman-Staffilani \cite{KS}.  We show that only the boundedness of the momentum support of $f$ \emph{after projecting to any two dimensional plane} is needed  for $(f, E, B)$ to remain $C^1$. 
\end{abstract}

\maketitle

\section{Introduction}

We consider the initial value problem for the relativistic Vlasov-Maxwell system in three dimensions. Let the particle density $f:\mathbb R_t \times \mathbb R_x^3\times \Rt_\pel \to \mathbb R_+$ be a non-negative function of time $t\in \mathbb R$, position $x\in \mathbb R^3$ and momentum $\pel\in\mathbb R^3$ and $E,B:\mathbb R_t \times \mathbb R_x^3 \to \mathbb R^3$ be time-dependent vector fields on the position space $\mathbb R^3$.

The relativistic Vlasov-Maxwell system can be written as
\bea
& &\rd_t f+\vh\cdot\nabla_x f+ (E+\vh\times B)\cdot \nabla_\pel f = 0,\label{vlasov}\\
& &\rd_t E= \nabla_x \times B- j,\quad \rd_t B=-\nabla_x\times E,\label{maxwell}\\
& &\nabla_x\cdot E=\rho,\quad \nabla_x \cdot B=0.\label{constraints}
\eea
Here we have the charge
$$\rho(t,x) \eqdef 4\pi\int_{\Rt} f(t,x,\pel) d\pel,$$
and the current
$$j_i(t,x) \eqdef  4\pi \int_{\Rt} \vh_i f(t,x,\pel) d\pel, \quad (i=1,2,3).
$$
In these expressions we define 
\bea
\vh\eqdef \frac{\pel}{\pel_0}, \quad \pel_0\eqdef \sqrt{1+|\pel|^2}.\label{vh.def}
\eea
Notice that given initial data $f_0$, $E_0$, $B_0$ which satisfy the constraint equations \eqref{constraints}, they are propagated by the evolution equation \eqref{vlasov} and \eqref{maxwell} as long as the solution remains sufficiently regular. 

According to the relativistic Vlasov-Maxwell system \eqref{vlasov}-\eqref{constraints}, the particle density $f$ is transported along the characteristics $(X(t),\PL(t))$, which verify the following ordinary differential equations:
\bea
\frac{d}{dt}{X}(t)=\hat{\PL}(t),\quad \frac{d}{dt}\PL(t)= E(t,X(t))+\hat{\PL}(t)\times B(t,X(t)).\label{char}
\eea
These characteristics are combined with suitable initial conditions.

The global existence of solutions given sufficiently regular finite energy initial data remains an open problem. A key result of Glassey-Strauss \cite{GS86} shows that the solution remains $C^1$ as long as\footnote{The original work \cite{GS86} actually requires that \eqref{GS.condition} holds for all approximations of $f$ instead of only $f$ itself, and uses initial data with regularity $f_0\in C_c^1(\mathbb R^3_x\times \mathbb R^3_\pel)$, $E_0, B_0 \in C^2(\mathbb R^3_x)$. The assumption that \eqref{GS.condition} holds for all approximations of $f$ can be removed by a standard application of energy estimates (see, for example, \cite{LS}) but this requires slightly more regularity for the initial data (as in the statement of Theorem \ref{GS.theorem}). } the momentum support of $f$ remains bounded:

\begin{theorem}[Glassey-Strauss \cite{GS86}]\label{GS.theorem}
Consider initial data $(f_0(x,\pel),E_0(x),B_0(x))$ which satisfies the constraints \eqref{constraints} such that $f_0\in H^5(\mathbb R^3_x\times \mathbb R^3_\pel)$ with compact support in $(x, \pel )$, $E_0, B_0 \in H^5(\mathbb R^3_x)$ and such that the initial particle density is non-negative, i.e., $f_0\geq 0$. 
Let $(f,E,B)$ be the unique classical solution to \eqref{vlasov}-\eqref{constraints} in $[0,T)$. Assume that there exists a bounded continuous function $\kappa(t):[0,T)\to\mathbb R_+$ such that 
\begin{equation}\label{GS.condition}
f(t,x,\pel)=0\quad\mbox{for }|\pel|\geq \kappa(t), \quad \forall x \in \mathbb{R}^3.
\end{equation}
Then, there exists $\epsilon>0$ such that the solution extends uniquely in $C^1$ beyond $T$ to an interval $[0,T+\epsilon]$.
\end{theorem}
\begin{remark}
We note that as a consequence of the assumptions of the above theorem, the initial energy is bounded
\begin{equation}\label{f.energy}
\frac{1}{2}\int_{\mathbb R^3} (|E_0|^2+|B_0|^2) dx+4\pi\int_{\mathbb R^3} \int_{\mathbb R^3} \pel_0 f d\pel dx<\infty,
\end{equation}
the initial particle density satisfies
\begin{equation}\label{ini.bd}
||f_0||_{L_{x,\pel}^\infty}< \infty,
\end{equation}
and the momentum support is initially bounded as
\begin{equation}\label{ini.bd.mom}
\sup\{|\pel|\mbox{: there exists }x\in\mathbb R^3\mbox{ such that }f_0(x,\pel)\neq 0\} <\infty.
\end{equation}
We record the above bounds for the initial data explicitly as they will be useful for the argument of our main theorem below.
\end{remark}

In this paper, we extend the Glassey-Strauss criterion to a new continuation criterion requiring only the boundedness of the momentum support of $f$ \emph{after projecting to any two dimensional plane}. This is in contrast to the Glassey-Strauss criterion \cite{GS86} which requires the boundedness of the full momentum support.

\begin{theorem}\label{main.theorem.0}
Consider initial data $(f_0(x,\pel),E_0(x),B_0(x))$ which satisfies the constraints \eqref{constraints} such that $f_0\in H^5(\mathbb R^3_x\times \mathbb R^3_\pel)$ with compact support in $(x, \pel )$, $E_0, B_0 \in H^5(\mathbb R^3_x)$ and such that the initial particle density is non-negative, i.e., $f_0\geq 0$. 

Let $(f,E,B)$ be the unique classical solution to \eqref{vlasov}-\eqref{constraints} in $[0,T)$. Assume that there exists a plane $Q$ with $0\in Q\subset \mathbb R^3$ and a bounded continuous function $\kappa(t):[0,T) \to\mathbb R_+$ such that 
$$f(t,x,\pel)=0\quad\mbox{for }|\mathbb P_Q \pel|\geq \kappa(t), \quad \forall x \in \mathbb{R}^3,$$
where $\mathbb P_Q$ denotes the orthogonal projection to the 2-plane $Q$.
Then, there exists $\epsilon>0$ such that the solution extends uniquely in $C^1$ beyond $T$ to an interval $[0,T+\epsilon]$.
\end{theorem}

\begin{remark}
The methods in the paper extend in a straight-forward way to an analogous continuation criterion for multi-species relativistic Vlasov-Maxwell system, as well as the case where the equation is coupled with a \emph{given} external electromagnetic force. We omit the details here.
\end{remark}

\begin{remark}
Since the relativistic Vlasov-Maxwell system is invariant under rotation, given any 2-plane $Q$, it can be rotated to coincide with the $\pel_1-\pel_2$ plane. Thus, without loss of generality, we can assume that the plane is given in the $(\pel_1,\pel_2,\pel_3)$ coordinate system by $\{\pel_3=0\}$. We will henceforth make this assumption.
\end{remark}

In fact, we prove a more general and quantitative theorem, which implies Theorem \ref{main.theorem.0}. More precisely:
\begin{theorem}\label{main.theorem}
Consider initial data $(f_0(x,\pel),E_0(x),B_0(x))$ which satisfies the constraints \eqref{constraints} such that $f_0\in H^5(\mathbb R^3_x\times \mathbb R^3_\pel)$ with compact support in $(x, \pel )$, $E_0, B_0 \in H^5(\mathbb R^3_x)$ and such that the initial particle density is non-negative, i.e., $f_0\geq 0$.

Let $(f,E,B)$ be the unique solution to \eqref{vlasov}-\eqref{constraints} in $[0,T)$. Assume that for every $t\in [0,T)$, there exists a measurable positive function $\kappa=\kappa(t,\gamma)>1$ such that
\bea
\sup\left\{\sqrt{\pel_1^2+\pel_2^2} : \frac{\pel_2}{\pel_1}=\tan\gamma\mbox{ and } f(t,x,\pel)\neq 0\mbox{ for some }x\in\Rt\right\} <\kappa,\label{main.assumption.1}
\eea
and for 
\bea
A(t) \eqdef \|\kappa(t,\cdot)\|_{L^4_\gamma},
\label{main.assumption.2}
\eea
we have
\bea
\int_0^T \left(A(t)^2+\left(\int_0^t A(s)^8 ds\right)^{\frac 12}\right)  dt<+\infty.\label{main.assumption.3}
\eea
Then, there exists $\epsilon>0$ such that the solution extends uniquely in $C^1$ beyond $T$ to an interval $[0,T+\epsilon]$.
\end{theorem}

\begin{remark}
Notice that $\kappa(t)$ from Theorem \ref{main.theorem.0} is bounded and thus satisfies the assumptions \eqref{main.assumption.1} - \eqref{main.assumption.3}.
\end{remark}

Our result is motivated in part by the fact that global regularity of the relativistic Vlasov-Maxwell system is known in lower dimensions and also under symmetry assumptions (see the discussions in Section \ref{sec.previous.results} later on). Here, we show that instead of assuming that the initial data are of lower dimensions in the spatial variables, we can also obtain an improved continuation criterion which makes an \emph{a priori} assumption that is lower dimensional in the momentum variables compared to the known results.

While in general no bounds of the quantity $\|\kappa\|_{L^p_t L^q_{\gamma}}$ are to our knowledge currently available a priori, we note that in the two-and-one-half dimensional case \cite{GS2.5D} Glassey-Schaeffer established that the orthogonal projection of the momentum support to a line is bounded, which implies\footnote{To see \eqref{kappa.ex}, notice that Glassey-Schaeffer showed that in the two-and-one-half dimensional case where $x_1$ is the direction of symmetry, 
$$P_1(T)=\sup\{|\pel_1|:f(t,x,\pel_1,\pel_2,\pel_3)\neq 0 \mbox{ for some }t\leq T, x\in\mathbb R^3,\pel_2,\pel_3\in\mathbb R\}\leq C_T$$
for some $C_T\geq 0$. Then, according to \eqref{main.assumption.1}, we can choose $\kappa(\gamma)=\frac{C_T}{\cos\gamma}+2$, which is in $L^{1,\infty}_{\gamma}$.
} in particular that
\begin{equation}\label{kappa.ex}
\|\kappa\|_{L^\infty_t([0,T];L^{1,\infty}_\gamma)}\leq C_T,
\end{equation}
where $L^{1,\infty}_\gamma$ denotes the usual weak-$L^1$ space in the variable $\gamma$. On the other hand, in order to obtain global regularity in the two-and-one-half dimensional case, the full strength of the assumption on two-and-one-half dimensionality, in addition to \eqref{kappa.ex}, was exploited in \cite{GS2.5D}. Without assuming any symmetries, Theorem \ref{main.theorem} shows that it suffices to have $\kappa$ in $L^4_{\gamma}$ in order to continue the solution.

\subsection{Some previous results on the relativistic Vlasov-Maxwell system}\label{sec.previous.results}

The existence of global in time solutions to the relativistic Vlasov-Maxwell system is only known in the perturbation regime or under symmetry assumptions. Small data global existence was proved by Glassey-Strauss in \cite{GSdilute}. This result was later generalized to the case where the initial data were only assumed to be nearly neutral \cite{GSnearneutral}. The small data case was also generalized in \cite{S} to include small initial data without assuming initial compact momentum support.

Global existence of classical solutions was established by Glassey-Schaeffer \cite{GS2.5D} for initial data with translational symmetry in one of the position variables, i.e., the so-called two-and-one-half dimensional case (see also the related global existence result in two dimensions in \cite{GS2D1}, \cite{GS2D2}). Global existence is also known for the class of spherically symmetric solutions, for which the relativistic Vlasov-Maxwell system reduces to the relativistic Vlasov-Poisson system \cite{GSsymmetric}. This class of solution is moreover shown to be globally stable \cite{Rein}.

While a large data well-posedness theory without symmetry assumptions is not available, various continuation criteria are known. The outstanding result is the aforementioned Glassey-Strauss criterion \cite{GS86}, which requires the momentum support to be bounded. Alternative proofs of the Glassey-Strauss theorem were subsequently given by Bouchut-Golse-Pallard \cite{BGP} and Klainerman-Staffilani \cite{KS}. In particular, in \cite{KS}, a Fourier-analytic approach was introduced to obtain the continuation criterion. 

Another type of continuation criteria assumes the boundedness of
$$M_{\th,q} \eqdef  ||\pZ^{\theta} f||_{L^q_x L^1_\pel},$$
It is known that for certain ranges of $\th$ and $q$, the boundedness of $M_{\th,q}$ implies the boundedness of the momentum support of $f$, and hence would be sufficient to conclude regularity by the Glassey-Strauss theorem. Results in this direction were first obtained by Glassey-Strauss in \cite{GS87}, \cite{GS87.2} for the $\th=1$, $q=+\infty$ case. The range $\th >\frac 4q$ for $6\leq q\leq +\infty$ was later achieved by Pallard \cite{Pallard}. An end-point case $\th=0$, $q=+\infty$ was subsequently attained by Sospedra--Alfonso-Illner \cite{AI}.

We also note that in the non-relativistic limit, the relativistic Vlasov-Maxwell system reduces to the non-relativistic Vlasov-Poisson system \cite{Schaeffer.limit}. In contrast to the relativistic Vlasov-Maxwell system, global regularity for the non-relativistic Vlasov-Poisson system in three dimensions has been established by Pfaffelmoser \cite{Pfaffelmoser}, Lions-Perthame \cite{LP}, and Schaeffer \cite{Schaeffer}.

\subsection{Strategy of proof}

To prove Theorem \ref{main.theorem}, we will reduce it to Theorem \ref{GS.theorem} by showing that the momentum support of $f$ is bounded. To this end, we make the following definition.

\begin{definition}
Let $P(t)$ be defined by
$$P(t) \eqdef 2+\sup\{|\pel|: f(s,x,\pel)\neq 0 \mbox{ for some } 0\leq s\leq t \mbox{ and }x\in \Rt\}.$$
\end{definition}

\begin{remark}
Notice that $P(t)$ is an upper bound of the momentum support in the past of time $t$. Moreover, $P(t)$ is by definition a non-decreasing function and $P(t)\geq 2$. 
\end{remark}

Our goal will thus be to provide a bound for $P(t)$. Notice that it is sufficient to establish an integral inequality of the form
\bea\label{goal.0}
P(t)\ls 1+\log P(t) \int_0^t g(s)P(s) ds,
\eea
where we use the notation that 
\bea 
x \ls y\mbox{  if there exists a bounded function } C(t)\mbox{ such that } x\leq C(t)y
\notag
\eea
and $g(s)$ is an integrable function.

In view of the ordinary differential equation \eqref{char} for the characteristics $\PL$, we in turn need to bound the integral of the electromagnetic fields $E$ and $B$ along characteristics of the Vlasov equation.

\subsubsection{Estimates for the singularity}

In \cite{GS86}, Glassey-Strauss found a representation of $E$ and $B$ as integrals on a backward light cone whose integrand depends only on $E$, $B$ and $f$, but not their derivatives. This in particular allowed them to prove Theorem \ref{GS.theorem}. However, this representation contains singular terms of the type
\begin{equation}\label{error.terms}
\int_{C_{t,x}} \int_{\Rt} \frac{(|E|+|B|)f}{(t-s)\pZ\sing} d\pel d\sigma,\quad \int_{C_{t,x}} \int_{\Rt} \frac{f}{(t-s)^2\pZ^2\sing^{\frac 32}} d\pel d\sigma.
\end{equation}
The above integral in position space is over a backward light cone $C_{t,x}$ emanating from $(t,x)$, i.e.,
\bea
C_{t,x} \eqdef \left\{(s,y)\in\mathbb R\times\Rt | ~0\le s\leq t, ~t-s=|y-x|\right\}.
\notag
\eea
We refer the readers to \eqref{normal.def} - \eqref{normal.coord} for the precise definition of this integral. 
Here, it suffices to note that $\om$ is a unit vector and the singular terms become large only if $\frac{\pel}{|\pel|}$ and  $\om$ are almost anti-parallel and $\pel$ is large.

Now the assumptions \eqref{main.assumption.1} - \eqref{main.assumption.3} in the main theorem are useful in two ways. On one hand, the assumptions \eqref{main.assumption.1} - \eqref{main.assumption.3} reduce the dimensionality of the integral since the range of integration in two of the three dimensions are a priori bounded\footnote{in an $L^4$-averaged sense along the angular directions $\gamma$.} by these assumptions.

On the other hand, while the pointwise behavior of the singularity is no better than
$$\frac{1}{\sing}\ls P^2,$$
our assumptions imply that this upper bound can only be achieved in a relatively small set in physical space. In particular, the set of $\om$ such that this bound is achieved has area $O(P^{-2})$. More precisely, the assumptions \eqref{main.assumption.1} - \eqref{main.assumption.3} also allow us to control the strength of the singularity when $-\om$ is not parallel to $e_3$ (where $e_3=(0,0,1)$ is the standard basis vector). This is because the singular term is large if $\frac{\pel}{|\pel|}$ and $\om$ are anti-parallel and $\pel$ is large. In the case where $-\om$ is not parallel to $e_3$, $\frac{\pel}{|\pel|}$ and $-\om$ can only be parallel if $\frac{\pel}{|\pel|}$ is not parallel to $e_3$. However, in this case the assumptions \eqref{main.assumption.1} - \eqref{main.assumption.3} provide a bound for $\pZ$.  This can be appropriately quantified and it provides an estimate for the singular terms except when $-\om$ is almost parallel to $e_3$. 

\subsubsection{Estimates for the nonlinear terms}

Now, returning to the estimate for $E$ and $B$, we control the contribution from the term
$$\int_{C_{t,x}} \int_{\Rt} \frac{(|E|+|B|)f}{(t-s)\pZ\sing} d\pel d\sigma.$$

One of the important observations in our work, which is highly reminiscent of that of Glassey-Schaeffer \cite{GS2.5D}, \cite{GS2D1}, \cite{GS2D2} in the two dimensional case, is that for this term, the most singular contributions are always coupled with good components of the electromagnetic field, i.e., those components that can be controlled after integrating along the null cone by the conserved energy of the relativistic Vlasov-Maxwell system\footnote{Note that this conservation law misses two of the six independent components of the electromagnetic field.}. More precisely, the most singular terms are of the form 
$$\int_{C_{t,x}} \int_{\Rt} \frac{|K_g|f}{(t-s)\pZ\sing} d\pel d\sigma,$$
where by energy conservation, $K_g$ (which is defined in \eqref{good.K} later on) satisfies the a priori assumption (as in Proposition \ref{cons.law.2}) that
$$\int_{C_{t,x}} |K_g|^2 d\sigma \lesssim 1.$$
Hence, we only need to control 
$$\int_{\Rt} \frac{f}{(t-s)\pZ\sing} d\pel$$
in $L^2$ on the cone $C_{t,x}$. Precisely because we are only required to obtain an $L^2$ bound (as opposed to an $L^\infty$ bound), we can take advantage of the fact that the singularity is large only on a small set of angular directions as mentioned above. This is sufficient to show that $L^2$ norm of this integral is bounded by $P\log P.$

However, unlike in the two-dimensional case, there are singular terms that are coupled with a ``bad" component of the electromagnetic field. Nevertheless, these terms are less singular. More precisely, they take the form
$$\int_{C_{t,x}} \int_{\Rt} \frac{(|E|+|B|)f}{(t-s)\pZ\sing^{\frac 12}} d\pel d\sigma.$$
Unlike for the previous terms, $E$ and $B$ cannot be controlled by the conservation law after integrating in $L^2$ along the cone $C_{t,x}$.

Nevertheless, in order to control the growth of the velocity support, it is sufficient to bound $E$ and $B$ after integration along a characteristic of the Vlasov equation. We apply an estimate of Pallard in \cite{Pallard}, which shows that after this integration, and up to an logarithmic loss in $P$, it suffices to control $\int_{\Rt}\frac{(|E|+|B|)f}{\pZ{\sing}^{\frac 12}} d\pel$ in $L^2$ of $\Rt$. Thus, we can use the conservation law for $E$ and $B$ in $L^2(\Rt)$ and it remains to bound $\int_{\Rt}\frac{f}{\pZ{\sing}^{\frac 12}} d\pel$ in $L^\infty$. This final goal can then be achieved by noticing that assumptions \eqref{main.assumption.1} - \eqref{main.assumption.3} give an a priori control over the size of the domain of integration in the $\pel_1-\pel_2$ plane.

\subsubsection{Estimates for the linear terms}

As pointed out above, there is another contribution to $E$ and $B$ from the term
$$\int_{C_{t,x}} \int_{\Rt} \frac{f}{(t-s)^2\pZ^2\sing^{\frac 32}} d\pel d\sigma.$$
While the function $\frac{1}{\pZ^2\sing^{\frac 32}}$ can be as large as $P$ on the support of $f$, it is much smaller ($\sim \frac{1}{\pZ^2}$) in the complement of a small set. Therefore, this term in fact behaves much more favorably and only grows logarithmically in terms of $P$. Combining all the above estimates, we therefore achieve \eqref{goal.0}.

\subsection{Outline of the paper}

We outline the remainder of the paper. In Section \ref{sec.cons.law}, we state the conservation laws that we will apply in the proof of the main theorem. In Section 3, we state the Glassey-Strauss representation of the electromagnetic field and observe the structures of the singular terms as indicated above. In Section 4, we apply the assumptions \eqref{main.assumption.1} - \eqref{main.assumption.3} to obtain the main estimates for the particle density. This will then be applied in Section 5 to derive the necessary bounds for the electromagnetic fields $E$ and $B$. Finally, in Section 6, we gather all the estimates to conclude the proof of Theorem \ref{main.theorem}.

In the rest of the paper we assume sometimes without additional commentary that $(f, E, B)$ is a smooth solution to the relativistic Vlasov-Maxwell system \eqref{vlasov}-\eqref{constraints} which satisfies the assumptions of Theorem \ref{main.theorem}.

\section{Conservation Laws}\label{sec.cons.law}

The solution to the relativistic Vlasov-Maxwell system obeys a pointwise identity:
\bea
\frac{\rd}{\rd t} e+\sum_{k=1}^3 \frac{\rd}{\rd x_k}\left((E\times B)_k + 4\pi\int_{\mathbb R^3} \pel_k  f d\pel\right)=0, \label{cons.id}
\eea
where the energy density $e$ is given by
$$e \eqdef \frac 12 (|E|^2+|B|^2)+4\pi\int_{\mathbb R^3} \pZ f d\pel,$$
with $\pZ$ defined as before in \eqref{vh.def}.

The identity \eqref{cons.id} can be integrated on spacetime regions to obtain conservation laws. 
We will derive two conservation laws that we will use later on. 
First, we integrate in the spacetime region bounded in the past by the initial slice $\{0\}\times \mathbb R^3$ and in the future by a constant time slice $\{t\}\times \mathbb R^3$. Since the initial energy is bounded by the assumptions of Theorem \ref{main.theorem} (see equation \eqref{f.energy}), we obtain

\begin{proposition}\label{cons.law.1}  Solutions to the relativistic Vlasov-Maxwell system \eqref{vlasov}-\eqref{constraints} satisfy
\bea
\frac 12 \int_{\{t\}\times \Rt} (|E|^2+|B|^2)dx+4\pi\int_{\{t\}\times\Rt\times\Rt} \pZ f d\pel dx =\mbox{ constant}.
\notag
\eea
\end{proposition}
For the second conservation law, we need to control the flux of the electromagnetic field integrated along a backward null cone. To this end, we integrate \eqref{cons.id} in the spacetime region bounded in the past by the initial slice $\{0\}\times \mathbb R^3$ and in the future by the backward null cone $C_{t,x}$ emanating from $(t,x)$, which is defined to be the set from \eqref{light.cone}.

Fix a point $(s,y)\in C_{t,x}$, where we recall that:
\bea
C_{t,x} = \left\{(s,y)\in\mathbb R\times\Rt | ~0\le s\leq t, ~t-s=|y-x|\right\}.
\label{light.cone}
\eea
Denote by $\om$ the outward normal to the 2-sphere $C_{t,x}\cap (\{s\}\times\Rt)$, i.e., 
\bea
\om \eqdef \frac{y-x}{|y-x|}.\label{normal.def}
\eea
The volume form on $C_{t,x}$ can be given in polar coordinates by
\bea
\int_{C_{t,x}} f ~d\sigma 
= 
\int_0^t ds  \int_0^{2\pi}  d\phi  \int_0^{\pi}  (t-s)^2\sin\th d\th  ~  f(s,x+(t-s)\om), \label{cone.vol.form}
\eea
where $\om$ takes the form
\bea
\om = (\sin\th\cos\phi,\sin\th\sin\phi,\cos\th)\label{normal.coord}
\eea
in this coordinate system.

We compute the flux of the electromagnetic field on the null cone, i.e., the boundary term $C_{t,x}$ arising from integrating \eqref{cons.id} by parts. We show that it is non-negative and moreover controls certain components of $E$ and $B$.  Notice\footnote{Here we will use the vector identity
$
a\cdot (b\times c) = b\cdot (c\times a) = c\cdot (a \times b)
$ for $a$, $b$, $c$ three vectors.}
 that
\begin{multline} \notag
\frac 12 \left(|E|^2+|B|^2\right)+\om\cdot\left(E\times B\right)
\\
=\frac 18 \left(2|E\cdot\om|^2+2|B\cdot\om|^2+|E+\om\times B|^2+|E-\om\times B|^2
+|B+\om\times E|^2\right.
\\
\left.+|B-\om\times E|^2-4(\om\times B)\cdot E+ 4(\om\times E)\cdot B\right)
\\
=\frac 14 \left(|E\cdot\om|^2+|B\cdot\om|^2+|E-\om\times B|^2+|B+\om\times E|^2\right).
\end{multline}
This motivates the following definition
\begin{equation}\label{good.K}
K_g^2 \eqdef |E\cdot\om|^2+|B\cdot\om|^2+|E-\om\times B|^2+|B+\om\times E|^2.
\end{equation}
Therefore, by the boundedness of the initial energy, we have proved

\begin{proposition}\label{cons.law.2}  
Solutions to the relativistic Vlasov-Maxwell system \eqref{vlasov}-\eqref{constraints} satisfy
$$
 \frac 14 \int_{C_{t,x}} K_g^2 ~d\sigma
+4\pi \int_{C_{t,x}} \int_{\Rt}  \pZ(1+\vh\cdot\om) f d\pel d\sigma 
 \leq \mbox{ constant}.
$$
\end{proposition}
Notice that this does not control all the components of $E$ and $B$. As we will see in the later sections, this conservation law will nevertheless be useful in controlling the most singular terms.

Finally, we also need the conservation law for the $L^\infty$ norm of the particle density $f$. This follows from integrating $f$ along the characteristics given by \eqref{char}. Since $f$ is initially bounded by \eqref{ini.bd}, we have
\begin{lemma}\label{cons.law.3}
$
 ||f||_{L^\infty_\pel L^\infty_x}\leq\mbox{ constant}.
$
\end{lemma}

This completes the statements of all the conservation law estimates that we will use in the later sections.  In the next section we decompose and estimate the electromagnetic fields.

\section{Glassey-Strauss decomposition of the electromagnetic fields}

In order to close our estimates, we need to obtain good control on the electromagnetic fields $E$ and $B$ under the a priori assumptions \eqref{main.assumption.1} - \eqref{main.assumption.3}. To this end, we need to exploit certain structures in the equations. Following \cite{GS86}, we decompose $E$ and $B$ in terms of:
$$4\pi E(t,x)=4\pi E=(E)_0+E_S+E_T,$$
and
$$4\pi B(t,x)=4\pi B=(B)_0+B_S+B_T,$$
where $(E)_0$ and $(B)_0$ depend only on the initial data and the other terms of $E$ are
\bea\label{ET.id}
E_T^i 
&=& -\int_{C_{t,x}}\int_{\Rt} \frac{(\om_i+\vh_i)(1-|\vh|^2)}{(t-s)^2(1+\vh\cdot\om)^2}f ~d\pel\, d\sigma,
\\
\label{ES.id}
E_S^i &=& -\int_{C_{t,x}}\int_{\Rt} \left(\frac{\delta_{ij}-\vh_i\vh_j}{1+\vh\cdot\om}\right)\frac{(E+\vh\times B)_j}{\pZ(t-s)} f d\pel\, d\sigma
\\
\notag
&& +\int_{C_{t,x}}\int_{\Rt} \left(\frac{(\om_i+\vh_i)(\om_j-(\om\cdot\vh)\vh_j)}{(1+\vh\cdot\om)^2}\right)\frac{(E+\vh\times B)_j}{\pZ(t-s)} f d\pel\, d\sigma,
\eea
for $i,j=1,2,3$, where we used the convention that repeated indices are summed over. 
The rest of the $B$ terms are given by
\bea
\label{BT.id}
B_T^i (t,x)&=& \int_{C_{t,x}}\int_{\Rt} \frac{(\om\times\vh)_i(1-|\vh|^2)}{(t-s)^2(1+\vh\cdot\om)^2}f d\pel\, d\sigma,
\\
\label{BS.id}
B_S^i (t,x)&=& -\int_{C_{t,x}}\int_{\Rt} \frac{\big(\om\times(E+\vh\times B)\big)_i}{\pZ(t-s)(1+\vh\cdot\om)} f d\pel\, d\sigma 
\\
\notag
&& +\int_{C_{t,x}}\int_{\Rt} \frac{(\om\times\vh)_i(\vh\cdot (E+\vh\times B))}{\pZ(t-s)(1+\vh\cdot\om)} f d\pel\, d\sigma 
\\
& &-\int_{C_{t,x}}\int_{\Rt} \frac{(\om\times \vh)_i(\om_j-(\om\cdot\vh)\vh_j)}{(1+\vh\cdot\om)^2}\frac{(E+\vh\times B)_j}{\pZ(t-s)} f d\pel\, d\sigma.
\notag
\eea
In these terms, the integration over the cone $C_{t,x}$ can be given in local coordinates as in \eqref{cone.vol.form}. We refer the readers to Theorem 3 in \cite{GS86} for a proof of this decomposition. 

In the remainder of this section, we prove bounds on the integral kernels in each of these terms. We first estimate the kernels in the terms $E_T$ and $B_T$:
\begin{proposition}\label{T.prop}  We have the following estimate
$$
|E_T(t,x)|+|B_T(t,x)|
\ls 
\int_{C_{t,x}}\int_{\Rt} \frac{f(s, x+(t-s)\omega, \pel)}{(t-s)^2\pZ^2(1+\vh\cdot\om)^{\frac 32}} d\pel\, d\sigma.
$$
\end{proposition}
\begin{proof}
Notice that 
\bea
1-|\vh|^2=\frac{1}{\pZ^2},\label{basic.ineq.1}
\eea
and 
\bea
(\om_i+\vh_i)^2 \leq |\om+\vh|^2 \leq 2+2\vh\cdot\om= 2\sing.\label{basic.ineq.2}
\eea
Thus, we have
$$\frac{(\om_i+\vh_i)(1-|\vh|^2)}{\sing^2}\leq \frac{2}{\pZ^2\sing^{\frac 32}}.$$
Therefore, by \eqref{ET.id},
$$|E_T|\ls \int_{C_{t,x}}\int_{\Rt} \frac{f}{(t-s)^2\pZ^2(1+\vh\cdot\om)^{\frac 32}} d\pel\, d\sigma.$$
To derive an analogous control for $B_T$, we first need to bound $|\vh\times\om|$. To this end, we consider
$$
\pZ(1+\vh\cdot\om)=\sqrt{1+|\pel|^2}+\pel\cdot\om
=\frac{1+|\pel|^2-(\pel\cdot\om)^2}{\pZ-\pel\cdot\om}
\geq\frac{1+|\pel\times\om|^2}{2\pZ}.
$$
This implies,
\bea
|\vh\times\om|^2\leq 2(1+\vh\cdot\om).\label{basic.ineq.3}
\eea
Therefore \eqref{BT.id}, \eqref{basic.ineq.1} and \eqref{basic.ineq.3} together  imply
$$|B_T|\ls \int_{C_{t,x}}\int_{\Rt} \frac{f}{(t-s)^2\pZ^2(1+\vh\cdot\om)^{\frac 32}} d\pel\, d\sigma.$$
This completes the proof.
\end{proof}

We next control the kernels in the terms $E_S$ and $B_S$:
\begin{proposition}\label{S.prop.1}  We have the following estimates
$$|E_S(t,x)|+|B_S(t,x)|\ls 
\int_{C_{t,x}}\int_{\Rt} \frac{|E+\vh\times B|}{(t-s)\pZ(1+\vh\cdot\om)} f ~ d\pel d\sigma.$$
\end{proposition}

\begin{proof}
By the definition of $E_S$ and $B_S$ in \eqref{ES.id} and \eqref{BS.id}, it suffices to show that
\bea
|(\vh_i+\om_i)(\om_j-(\om\cdot\vh)\vh_j)|\ls (1+\vh\cdot\om),\label{ES.1}
\eea
and
\bea
|(\vh\times\om)_i(\om_j-(\om\cdot\vh)\vh_j)|\ls (1+\vh\cdot\om).\label{BS.1}
\eea
We have
\beaa
|(\om_j-(\om\cdot\vh)\vh_j)|&=&|\om_j+\vh_j-(1+\vh\cdot\om)\vh_j|\\
&\leq& |\om_j+\vh_j|+|(1+\vh\cdot\om)\vh_j|.
\eeaa
Thus, by \eqref{basic.ineq.2}, we have
$$|(\om_j-(\om\cdot\vh)\vh_j)|\ls (1+\vh\cdot\om)^{\frac 12}.$$
Moreover, by \eqref{basic.ineq.2} and \eqref{basic.ineq.3}, we have
$$|\vh_i+\om_i|+|(\vh\times\om)_i|\ls (1+\vh\cdot\om)^{\frac 12}.$$
Thus \eqref{ES.1} and \eqref{BS.1} follow.
\end{proof}

We also need the following bound for the Lorentz force:
\begin{proposition}\label{S.prop.2}  The Lorentz force is bounded as follows.
$$|E+\vh\times B|\ls |E\cdot\om|+|B\cdot\om|+|B+\om\times E|+(1+\vh\cdot\om)^{\frac 12} |B|.$$
\end{proposition}

\begin{proof}
Clearly, we have
\bea
|E+\vh\times B|\leq |\om\cdot(E+\vh\times B)|+|\om\times(E+\vh\times B)|.
\label{Lorentz.0}
\eea
For the first term,
we have
\bea
|\om\cdot(E+\vh\times B)|\leq |\om\cdot E|+|\om\cdot(\vh\times B)|=|\om\cdot E|+|(\om\times\vh)\cdot B)|.\label{Lorentz.1}
\eea
By \eqref{basic.ineq.3}, we have
$$|\vh\times\om|\ls (1+\vh\cdot\om)^{\frac 12}.$$
Therefore, the bound derived in \eqref{Lorentz.1} is acceptable. 

We now estimate the second term\footnote{Here we will use the vector identity
$
a\times (b\times c) = b(a\cdot c) - c (a \cdot b)
$ for $a$, $b$, $c$ three vectors.}
 in \eqref{Lorentz.0}.
\beaa
|\om\times(E+\vh\times B)|&=&|\om\times E+B-(1+\vh\cdot\om)B+(\om\cdot B)\vh|\\
&\leq &|\om\times E+B|+(1+\vh\cdot\om)|B|+|\om\cdot B|,
\eeaa
and therefore all of the terms are acceptable.
\end{proof}
Propositions \ref{S.prop.1} and \ref{S.prop.2} together imply that the $E_S$ and $B_S$ terms can be decomposed into a less singular term, and a more singular term which is coupled only to the components of the electromagnetic field that can be controlled by the flux. More precisely,
\begin{proposition}\label{S.prop.3}
We can decompose 
$$E_S=E_{S,1}+E_{S,2}$$
and
$$B_S=B_{S,1}+B_{S,2}$$
such that 
$$
\left(|E_{S,1}|+|B_{S,1}|\right)(t,x)\ls \int_{C_{t,x}}\int_{\Rt} \frac{|B| f(s, x+(t-s)\omega, \pel)}{(t-s)\pZ(1+\vh\cdot\om)^{\frac 12}} d\pel\, d\sigma
$$
and 
$$
|E_{S,2}|+|B_{S,2}|\ls \int_{C_{t,x}}\int_{\Rt} \frac{\left( |E\cdot\om|+|B\cdot\om|+|B+\om\times E| \right) f}{(t-s)\pZ(1+\vh\cdot\om)} d\pel\, d\sigma,
$$
where the last upper bound is more singular but it contains the components of the electromagnetic field that can be controlled by the flux terms in Proposition \ref{cons.law.2}.
\end{proposition}

\section{Main estimates for the particle density}

One of the main challenges in obtaining control for $E$ and $B$ is the presence of the singularities in the kernels. In this section, we estimate the singularities under the assumptions \eqref{main.assumption.1} - \eqref{main.assumption.3}. This will allow us to obtain the main estimate in Proposition \ref{main.est} for the quantity $\int_{\Rt}\frac{f}{\pZ (1+\vh\cdot\om)} d\pel$.

Recall that $e_3=(0,0,1)$ is the standard basis vector. We will need to obtain improved estimates when  taking into account the angles between the vectors $\frac{\pel}{|\pel|}$, $\om$ and $e_3$. We define the following notation to denote angles between two unit vectors:

\begin{definition}
Let $V_1$, $V_2$ be two unit vectors in $\Rt$. Let $\angle(V_1,V_2)$ be the unique number in the range $[0,\pi]$ such that
$$V_1\cdot V_2 \eqdef \cos\left(\angle(V_1,V_2)\right).$$
In particular, the angles $\angle$ are always defined to be positive.
\end{definition}

We also define $\angle(V_1,\pm V_2)$ to be the minimum of the two angles, i.e.,
\begin{definition}
 $\angle(V_1,\pm V_2)\eqdef \min\{\angle(V_1, V_2),\angle(V_1,-V_2)\}.$
\end{definition}
\begin{remark}
In particular, we observe that
$$
\angle(V_1,\pm V_2)
=
\min\{\angle(V_1, V_2),\pi -\angle(V_1,V_2)\}
$$
and 
we thus have $\angle(V_1,\pm V_2)\in [0,\frac{\pi}{2}]$.
\end{remark}

We now proceed to derive bounds for the singularity.
First, we have the following trivial pointwise estimate for the singularity:
\begin{proposition}\label{trivial.sing.est}
$
{\sing}^{-1}\ls \min\{\pZ^2,(\angle(\frac{\pel}{|\pel|},-\om))^{-2}  \}.
$
\end{proposition}

\begin{proof}
We first show that the singularity can be estimated by the size of $|\pel|$:
\begin{multline}\notag
\frac{1}{\sing}\leq \frac{1}{1-|\vh|}
=\frac{\sqrt{1+|\pel|^2}}{\sqrt{1+|\pel|^2}-|\pel|} \cdot\frac{\sqrt{1+|\pel|^2}+|\pel|}{\sqrt{1+|\pel|^2}+|\pel|}
\\
\leq  2(1+|\pel|^2)
=2 \pZ^2.
\end{multline}
Notice that this in particular shows that this singularity in the direction anti-parallel to $\pel$ is bounded above by the square of the supremum of the momentum support.

We then show the bound by the squared inverse of the angle between $\frac{\pel}{|\pel|}$ and $\om$:
$$
\frac{1}{\sing}
= \frac{1}{1-|\vh| \cos (\angle(\frac{\pel}{|\pel|},-\om))}
\ls (\angle(\frac{\pel}{|\pel|},-\om))^{-2}.
$$
This completes the proof.
\end{proof}

Using the assumption \eqref{main.assumption.1} on the support of $f$, we have an a priori bound for the size of $|\pel|$. This will allow us to apply Proposition \ref{trivial.sing.est} to control the singularity.

\begin{proposition}\label{v.est}  On the support of $f$, we have the estimate
\bea
|\pel|\ls \frac{\kappa(t,\gamma(\pel))}{\angle(\frac{\pel}{|\pel|},\pm e_3)},\notag
\eea
where $\gamma=\gamma(\pel)$ is defined implicitly by
$$(\pel_1,\pel_2,\pel_3)=\left(\sqrt{\pel_1^2+\pel_2^2}\cos\gamma, \sqrt{\pel_1^2+\pel_2^2}\sin\gamma,\pel_3\right).$$
\end{proposition}
\begin{proof}
The main assumption \eqref{main.assumption.1} of Theorem \ref{main.theorem}  implies that
$$\sqrt{\pel_1^2+\pel_2^2}\leq \kappa(t,\gamma(\pel))$$
on the support of $f$.

The angle $\angle(\frac{\pel}{|\pel|},\pm e_3)$ is given by
$$\sin(\angle(\frac{\pel}{|\pel|},e_3))=\sin(\angle(\frac{\pel}{|\pel|},- e_3))=\frac{\sqrt{\pel_1^2+\pel_2^2}}{\sqrt{\pel_1^2+\pel_2^2+\pel_3^2}}.$$
Therefore, since $\angle(\frac{\pel}{|\pel|}, \pm e_3)\in [0,\frac{\pi}{2}]$, we have
$$\sqrt{\pel_1^2+\pel_2^2+\pel_3^2}\leq \frac{\kappa(t,\gamma(\pel))}{\sin(\angle(\frac{\pel}{|\pel|},\pm e_3))}
\ls 
\frac{\kappa(t,\gamma(\pel))}{\angle(\frac{\pel}{|\pel|},\pm e_3)},$$
as claimed.
\end{proof}

The estimates in Propositions \ref{trivial.sing.est} and \ref{v.est} imply 

\begin{proposition}\label{sing.est}  We have the following estimate
$${\sing}^{-1}\ls \min\big\{\kappa^2(t,\gamma(\pel))(\angle(\frac{\pel}{|\pel|},\pm e_3))^{-2},(\angle(\frac{\pel}{|\pel|},\pm \om))^{-2}  \big\}.$$
\end{proposition}

\begin{proof}
By Proposition \ref{v.est}, we have that
$$
\pZ \ls \frac{\kappa(t,\gamma(\pel))}{\angle(\frac{\pel}{|\pel|},\pm e_3)}.
$$
The conclusion then follows from Proposition \ref{trivial.sing.est}.
\end{proof}

With the preparatory bounds above, we can now state and prove our main estimates for the particle density. This will then be used to control the electromagnetic fields in the next section.

\begin{proposition}[Main estimate]\label{main.est}  Recalling \eqref{main.assumption.2} for $r \ge 0$ we have the estimate
\bea
\int \frac{f(t,x+r\om,\pel)}{\pZ \sing} d\pel \ls 
\min\left\{P(t)^2\log P(t), \frac{A(t)^4\log P(t)}{(\angle(e_3,\pm \om))^{2}} \right\}.
\eea
\end{proposition}
\begin{proof}
In the proof below, for notational purposes, we will suppress the explicit dependence of $P(t)$ on $t$ and write $P$ instead.

We first show that 
$$\int \frac{f}{\pZ \sing} d\pel \ls P^2\log P,\quad\mbox{for every }\om.$$
To this end, for every fixed $\om$, 
we use polar coordinates $(\theta',\phi')$ such that $\theta'$ is chosen for $-\om$ to lie on the $\theta'=0$ half-axis, i.e.,
$$\pel\cdot (-\om)=|\pel| \cos\theta'. $$
(This choice of the polar coordinate system is different from that in \eqref{cone.vol.form}.)

In the $(|\pel|,\theta',\phi')$ coordinate system, we can bound using Proposition \ref{trivial.sing.est}
$$
{\sing}^{-1}\ls \min\{\pZ^2,(\angle(\frac{\pel}{|\pel|}, -\om))^{-2}\}\ls \min\{\pZ ^2,(\th')^{-2}\}.
$$
We will use the bound $\frac{1}{\sing}\ls \pZ ^2$ for $0\leq \theta'\leq P^{-1}$ and use $\frac{1}{\sing}\ls (\theta')^{-2}$ for $P^{-1}\leq \theta' \leq \pi$. More precisely, by the conservation law in Lemma \ref{cons.law.3},
we have
\begin{multline} \notag
\left|\int \frac{f}{\pZ  \sing} d\pel \right|
\ls \int_{|\pel|\leq P} \frac{1}{\pZ  \sing} d\pel
\\
\ls \int_0^P |\pel|^2 d|\pel|  \int_0^{\pi} \sin\theta' d\theta'  \int_0^{2\pi} d\phi'  \frac{1}{\pZ  \sing} 
\\
\ls \int_0^P |\pel| d|\pel| \int_{P^{-1}}^{\pi} (\th')^{-2} \sin\th'  d\th' 
 +
\int_0^P (|\pel|+|\pel|^3)d|\pel|  \int_{0}^{P^{-1}} \sin\th'  d\th' 
\\
\ls P^2\log P,
\end{multline}
as desired.

We will now show that
\begin{equation}\label{clam.est}
\int \frac{f}{\pZ  \sing} d\pel \ls 
A(t)^2 \log P(\angle(e_3,\pm\om))^{-2}
+A(t)^4 (\angle(e_3,\pm\om))^{-2}.
\end{equation}
Notice that we assumed $\kappa>1$ which guarantees $A(t)^2 \ls A(t)^4$.  
For notational convenience, let $\beta \eqdef \angle(e_3,\pm\om)=\min\{\angle(e_3,\om),\angle(e_3,-\om)\}$. In particular, $\beta\leq \frac{\pi}{2}$.
 
To estimate $\int \frac{f}{\pZ  \sing} d\pel$, we divide the domain of integration into the following regions according to the size of the angles $\angle(\frac{\pel}{|\pel|},\pm e_3)$:
$$
I \eqdef \left\{\angle(\frac{\pel}{|\pel|},e_3)\leq \frac{\beta}{2}\right\}\cup\left\{\angle(\frac{\pel}{|\pel|},-e_3)\leq \frac{\beta}{2}\right\},
$$
$$II \eqdef \left\{\angle(\frac{\pel}{|\pel|},e_3)\geq \frac{\beta}{2}\right\}\cap\left\{\angle(\frac{\pel}{|\pel|},-e_3)\geq \frac{\beta}{2}\right\}.
$$
We will first estimate region $I$.

On region $I$, 
in the case $\angle(\frac{\pel}{|\pel|},e_3)\leq \frac{\beta}{2}$, we have by the triangle inequality,
$$\angle(\frac{\pel}{|\pel|},\om)\geq |\angle(\om,e_3)-\angle(\frac{\pel}{|\pel|},e_3)|\geq\frac{\beta}{2}.$$
In the case $\angle(\frac{\pel}{|\pel|},-e_3)\leq \frac{\beta}{2}$, we again have, by the triangle inequality,
$$\angle(\frac{\pel}{|\pel|},\om)\geq |\angle(\om,-e_3)-\angle(\frac{\pel}{|\pel|},-e_3)|\geq\frac{\beta}{2}.$$
Similarly, we have
$$\angle(\frac{\pel}{|\pel|},-\om)\geq \frac{\beta}{2}.$$
Thus, by Proposition \ref{sing.est}, we have that
\bea
\frac{1}{\sing}\ls \beta^{-2}. \label{I.sing.bound}
\eea
Therefore, estimating the singularity pointwise by \eqref{I.sing.bound} and bounding the remaining integral in Cartesian coordinates,
we have
\begin{multline} \label{I.est}
\int_I \frac{f}{\pZ  \sing} d\pel 
\ls  \beta^{-2}\int_{D} \frac{1}{\pZ } d\pel
\\
\ls  \beta^{-2}\int_{-P}^{P}   \frac{d\pel_3}{\sqrt{1+\pel_3^2}}   \iint_{D'}d\pel_1 d\pel_2 
\ls  A(t)^2 \beta^{-2} \log P,
\end{multline}
where $D$ is the subset of $\mathbb R^3$ given by 
$$D\eqdef \{p: f(t,x,\pel)\neq 0 \mbox{ for some }x\in \mathbb R^3\},$$
$D'$ is the subset of $\mathbb R^2$ given by
$$D'\eqdef \{(p_1,p_2): f(t,x,p_1,p_2,p_3)\neq 0 \mbox{ for some }x\in \mathbb R^3,\, \pel_3\in \mathbb R\}.$$
In \eqref{I.est}, we have estimated the integral $\iint_{D'} d\pel_1 d\pel_2$ by changing into the 2-dimensional polar coordinates:
$$\iint_{D'} d\pel_1 d\pel_2 = \int_0^{2\pi} \int_0^{\kappa(t,\gamma)}  u du d\gamma \leq \sqrt{2\pi}||\kappa(t,\cdot)||_{L^4_\gamma}^2\ls A(t)^2.$$
Then \eqref{I.est} establishes a better estimate than \eqref{clam.est} in region $I$.

We now move on to the estimate for the integral in region $II$. To this end, we use a system of polar coordinates (defined differently from above but the same as in \eqref{cone.vol.form}) with axis in the direction of $e_3$, i.e.,
$$\pel\cdot e_3=|\pel| \cos\theta. $$
 In other words, $\th= \angle(\frac{\pel}{|\pel|},e_3)$. Thus, in region $II$, we have, by definition
 $$\frac{\beta}{2}\leq \th\leq \pi-\frac{\beta}{2}.$$
By definition of these polar coordinates, $\phi$ coincides with $\gamma(\pel)$. By Proposition \ref{v.est}, 
$$|\pel|\ls {\kappa(t,\phi)}({\th}^{-1}+(\pi-\th)^{-1}).$$
We therefore have using also Proposition \ref{sing.est} that
\begin{multline} \notag
\int_{II} \frac{f}{\pZ  \sing} d\pel \\
\ls  
\int_0^{2 \pi} d\phi \int_{\frac{\beta}{2}}^{\pi-\frac{\beta}{2}} \sin\th  d\th  \int_0^{C\kappa(t,\phi)({\th}^{-1}+(\pi-\th)^{-1})} |\pel|  d|\pel| ~{\kappa(t,\phi)}^2({\th}^{-2}+(\pi-\th)^{-2}) 
\\
\ls  \int_0^{2 \pi} d\phi \int_{\frac{\beta}{2}}^{\pi-\frac{\beta}{2}} d\th ~ (\kappa(t,\phi))^4(\th^{-3}+(\pi-\th)^{-3})  
\\
\ls  \beta^{-2} \int_0^{2 \pi} (\kappa(t,\phi))^4 d\phi
\ls  A(t)^4\beta^{-2}.
\end{multline}
This completes the proof.
\end{proof}

\section{Estimates for the electromagnetic fields}
We now apply the estimates we obtained for the particle density to control the electromagnetic fields. We first bound $E_T$ and $B_T$:

\begin{proposition}\label{field.est.1} For $E_T$ from \eqref{ET.id} and $B_T$ from \eqref{BT.id} we have
$$
|E_T(t,x)|+|B_T(t,x)|\ls \log P(t)+(\log P(t))^2\int_0^t A(s)^4 ds.
$$
\end{proposition}

\begin{proof}
Recall from Proposition \ref{T.prop} that we have
\begin{multline} \label{T.basic.est}
|E_T(t,x)|+|B_T(t,x)|
\ls \int_{C_{t,x}}\int_{\Rt} \frac{f(s,x+(t-s)\om,\pel)}{(t-s)^2\pZ ^2(1+\vh\cdot\om)^{\frac 32}} d\pel\, d\sigma
\\
\ls \int_0^t ds \int_0^{2\pi}d\phi \int_0^{\pi} \sin\th d\th\int_{\Rt} d\pel \frac{f(s,x+(t-s)\om,\pel)}{\pZ ^2(1+\vh\cdot\om)^{\frac 32}},
\end{multline}
where we have adopted the coordinate system in \eqref{cone.vol.form}.
We recall that in this coordinate system
$\om= (\sin\th\cos\phi,\sin\th\sin\phi,\cos\th)$ as in \eqref{normal.coord}.

By Proposition \ref{trivial.sing.est} we have ${\sing}^{-1}\ls \pZ ^2$, so that we can
use the main estimates in Proposition \ref{main.est} to obtain
\begin{multline} \notag
\int_{\Rt} \frac{f(s,x+(t-s)\om,\pel)}{\pZ ^2(1+\vh\cdot\om)^{\frac 32}} d\pel
\ls \int_{\Rt} \frac{f(s,x+(t-s)\om,\pel)}{\pZ (1+\vh\cdot\om)} d\pel\\
\ls \min\{P(t)^2\log P(t), A(t)^4\log P(t)(\angle(e_3,\pm\om))^{-2} \}.
\end{multline}
Returning to \eqref{T.basic.est}, we have
\begin{multline} \notag
(|E_T|+|B_T|)(t,x)
\ls 
\int_0^t \int_0^{2\pi} \int_{0}^{P(t)^{-1}} P(s)^2\log P(s) \sin\th d\th\, d\phi\, ds
\\
+\int_0^t \int_0^{2\pi} \int_{\pi-P(t)^{-1}}^{\pi} P(s)^2\log P(s) \sin\th d\th\, d\phi\, ds\\
 +\int_0^t \int_0^{2\pi} \int_{P(t)^{-1}}^{\pi-P(t)^{-1}} A(s)^4(\th^{-2}+(\pi-\th)^{-2})\log P(s) \sin\th d\th\, d\phi\, ds
 \\
\ls  \log P(t)+\log P(t)\int_0^t  \int_0^{2\pi}\int_{P(t)^{-1}}^{\pi-P(t)^{-1}} A(s)^4( \th^{-1}+(\pi-\th)^{-1}) d\th\, d\phi\, ds
\\
\ls  \log P(t)+(\log P(t))^2\int_0^t A(s)^4 ds.
\end{multline}
This completes the proof.
\end{proof}

We now proceed to estimate $E_{S,1}$ and $B_{S,1}$. While these terms are less singular compared to $E_{S,2}$ and $B_{S,2}$, they are not coupled to a good component of the electromagnetic field; a component that can be controlled after integrating along a null cone by the conservation law. As a result, we are unable to use a direct estimate of the particle density to close the argument. Instead, we integrate $E_{S,1}$ and $B_{S,1}$ along the projection of a characteristic of the Vlasov equation to the position space. This will enable us to use the conservation law for $E$ and $B$ in $L^2(\Rt)$.

To this end, we apply an estimate of Pallard:
\begin{proposition}[Pallard {\cite[Lemma 1.3]{Pallard}}]\label{P.lemma}
Let $X(t):\mathbb R_t \to \mathbb R^3$ be a $C^1$ function  with $|X'(t)|< 1$ and define
$$
I \eqdef \int_0^t ds'  \int_{C_{s',x}} d\sigma(s,\omega)  ~ \frac{g(s,X(s')+(s'-s)\omega)}{(s'-s)}. 
$$
Then 
$$
I\ls \int_0^t ds ~ \left(\int_s^t ds' ~ \left(1+\left| \log\left(1-|X'(s')|\right)\right| \right) \right)^{\frac 12} ~ \|g(s,\cdot)\|_{L^2(\mathbb R^3)}.
$$
\end{proposition}

This proposition is slightly stronger than Lemma 1.3 in \cite{Pallard} in the sense that the original statement does not have both of the integrations in time, i.e., it reads
$$
I\ls \left(1+\left| \log(1-\sup_{0\leq s\leq t}|X'(s)|)\right|^{\frac 12}\right)\sup_{0\leq s\leq t}  ~ \|g(s,\cdot)\|_{L^2(\mathbb R^3)}.
$$
The stronger version in fact follows from the proof of Lemma 1.3 in \cite{Pallard}.  

We  point out for the reader  that we also use $X(s')+(s'-s)\omega$ and Pallard \cite{Pallard} uses the opposite sign $X(s')-(s'-s)\omega$.

\begin{proof}
Notice by Fubini's theorem that we have
\begin{equation*}
\begin{split}
I= &\int_0^t ds'\, \int_0^{s'} ds\, \int_0^{2\pi} d\phi\, \int_0^{\pi} (s'-s)\sin\th\, d\th\, g(s,X(s')+(s'-s)\omega)\\
=& \int_0^{t} ds\, \int_s^t ds'\,\int_0^{2\pi} d\phi\, \int_0^{\pi} (s'-s)\sin\th\, d\th\, g(s,X(s')+(s'-s)\omega).
\end{split}
\end{equation*}
We therefore define
\begin{equation*}
\tilde{I}_{s,t} 
\eqdef 
\int_s^t ds'\,\int_0^{2\pi} d\phi\, \int_0^{\pi} (s'-s)\sin\th\, d\th\, g(s,X(s')+(s'-s)\omega).
\end{equation*}
Then as in Lemma 2.1 in \cite{Pallard} we see that $\pi \eqdef X(s')+(s'-s)\omega$ is a $C^1_{s',\th,\phi}$ diffeomorphism with the Jacobian given by $J_\pi = (X'(s)\cdot \omega +1)(s'-s)^2\sin\theta$.  Then using the Cauchy-Schwartz inequality and this change of variable we have
\begin{multline}
\notag
\tilde{I}_{s,t} 
\ls 
\| g(s,\cdot)\|_{L^2(\mathbb{R}^3)}
\left( \int_s^t ds'\,\int_0^{2\pi} d\phi\, \int_0^{\pi}d\th ~ \frac{(s'-s)^2\sin^2\th}{\left| J_\pi \right|}  \right)^{1/2} 
\\
\ls 
\| g(s,\cdot)\|_{L^2(\mathbb{R}^3)}
\left(\int_s^t ds' ~ \left(1+\left| \log\left(1-|X'(s')|\right)\right| \right) \right)^{\frac 12}.
\end{multline}
The last inequality follows by evaluating the integral in $d\phi d\theta$ and estimating the result from above.
This implies the desired conclusion after integration in $s$.
\end{proof}

Using Proposition \ref{P.lemma}, we obtain the following bound on

\begin{proposition}\label{field.est.2}
Let $X(t)$ be a characteristic associated to the solution of the relativistic Vlasov-Maxwell system. Then
$$\int_0^t \left( |E_{S,1}(s,X(s))|+|B_{S,1}(s,X(s))| \right) ds \ls \sqrt{\log P(t)}\int_0^t A(s)^2 P(s) ds.$$
\end{proposition}
\begin{proof}
Since $X'(s)= \hat{\PL}(s)$ along a characteristic, we have
$$\sup_{s\leq s'\leq t} |X'(s')|\leq 1$$
and
$$\inf_{s\leq s'\leq t}\left( 1-|X'(s')|\right)\geq 1-\sup_{s\leq s'\leq t}|X'(s')|\gtrsim P(t)^{-2},$$
which imply
$$\left(\int_s^t ds' ~ \left(1+\left| \log\left(1-|X'(s')|\right)\right| \right) \right)^{\frac 12}\ls (t-s)^{\frac 12}\big(\log P(t)\big)^{\frac 12}.$$
Then we use the bound for $E_{S,1}$ and $B_{S,1}$ in Proposition \ref{S.prop.3} and Pallard's estimate in Proposition \ref{P.lemma} with 
$$
g(s, x+(t-s)\omega) = \int_{\Rt} \frac{|B| f(s, x+(t-s)\omega, \pel)}{\pZ (1-\vh\cdot\om)^{\frac 12}} d\pel.
$$
Thus it suffices to show that
$$
\left\|\int \frac{|B| f}{\pZ \sing^{\frac 12}} d\pel \right\|_{L^2(\mathbb R^3)} \ls A(s)^2 P(s). 
$$ 
Recall from Proposition \ref{trivial.sing.est} that
${\sing}^{-1}\ls \pZ ^2$.
Then we have
$$
\left\|\int \frac{|B| f}{\pZ \sing^{\frac 12}} d\pel \right\|_{L^2(\mathbb R^3)}\\
\ls \left\| B \right\|_{L^2(\mathbb R^3)}  \left\| \int f d\pel \right\|_{L^\infty(\mathbb R^3)}.
$$
By the conservation law in Proposition \ref{cons.law.1},
$$\left\| B \right\|_{L^2(\mathbb R^3)}\ls 1.$$
On the other hand, the particle density can be estimated by the total volume of the momentum support of $f$, i.e.,
$$
\left\|\int f d\pel \right\|_{L^\infty(\mathbb R^3)}
\ls  \int_{-P(s)}^{P(s)} d\pel_3 \int_0^{2\pi} d\gamma   \int_0^{\kappa(s,\gamma)}  u du 
\ls  A(s)^2 P(s).
$$
This completes the proof.
\end{proof}

Finally, the terms $E_{S,2}$ and $B_{S,2}$ can be controlled using the estimates in Proposition \ref{main.est}.

\begin{proposition}\label{field.est.3}  We have the following estimates
$$
|E_{S,2}(t,x)|+|B_{S,2}(t,x)|
\ls 
P(t)\log P(t)+P(t)\log P(t)\left(\int_0^t A(s)^8 ds\right)^{\frac 12}.
$$
\end{proposition}

\begin{proof}
We use the estimates from Proposition \ref{S.prop.3} and Cauchy-Schwarz to obtain
\begin{multline}\label{S.basic.est}
|E_{S,2}(t,x)|+|B_{S,2}(t,x)|
\ls 
\int_{C_{t,x}}\int_{\Rt} \frac{(\left| K_g \right| f)(s,x+(t-s)\om,\pel)}{(t-s)\pZ (1+\vh\cdot\om)} d\pel\, d\sigma
\\
\ls 
\int_0^t \int_0^{2\pi} \int_0^{\pi}\int_{\Rt} \frac{(\left| K_g \right| f)(s,x+(t-s)\om,\pel)}{\pZ (1+\vh\cdot\om)} d\pel\, (t-s) \sin\th d\th\, d\phi\, ds
\\
\ls \left(\int_0^t \int_0^{2\pi} \int_0^{\pi} |K_g|^2  \sin\th d\th\, d\phi\,(t-s)^2 ds\right)^{\frac 12}
\\
\times\left(\int_0^t \int_0^{2\pi} \int_0^{\pi}
\left(\int_{\Rt} \frac{f(s,x+(t-s)\om,\pel)}{\pZ (1+\vh\cdot\om)} d\pel
\right)^2\sin\th d\th\, d\phi\, ds\right)^{\frac 12}.
\end{multline}
Here we defined $K_g$ in \eqref{good.K}.

By the conservation law in Proposition \ref{cons.law.2},
$$\int_0^t \int_0^{2\pi} \int_0^{\pi} |K_g|^2  \sin\th d\th\, d\phi\, (t-s)^2 ds= ||K_g||_{L^2(C_{t,x})}^2 \ls 1.$$
The particle density term can be estimated using Proposition \ref{main.est}:
\beaa
& &\int_0^t \int_0^{2\pi} \int_0^{\pi}
\left(\int_{\Rt} \frac{f(s,x+(t-s)\om,\pel)}{\pZ (1+\vh\cdot\om)} d\pel
\right)^2\sin\th d\th\, d\phi\, ds
\\
&\ls &\int_0^t \int_0^{2\pi} \int_{0}^{P(s)^{-1}} P(s)^4\log^2 P(s) \sin\th d\th\, d\phi\, ds
\\
& & +\int_0^t \int_0^{2\pi} \int_{\pi-P(s)^{-1}}^{\pi} P(s)^4\log^2 P(s) \sin\th d\th\, d\phi\, ds
\\
& & +\int_0^t \int_0^{2\pi} \int_{P(s)^{-1}}^{\pi-P(s)^{-1}} A(s)^8(\th^{-4}+(\pi-\th)^{-4})\log^2 P(s) \sin\th d\th\, d\phi\, ds\\
&\ls & P(t)^2\log^2 P(t)+\log^2 P(t)\int_0^t A(s)^8 ds \int_{P(t)^{-1}}^{\pi-P(t)^{-1}} ( \th^{-3}+(\pi-\th)^{-3}) d\th\, \\
&\ls & P(t)^2\log^2 P(t)+P(t)^2\log^2 P(t)\int_0^t A(s)^8 ds.
\eeaa
Returning to \eqref{S.basic.est}, we thus conclude that the estimate in the statement of Proposition \ref{field.est.3} holds.
This completes the proof.
\end{proof}

\section{Conclusion of the proof}
We now conclude the proof of Theorem \ref{main.theorem}. By Theorem \ref{GS.theorem}, it suffices to show that $P$ is bounded by a function depending only on $t$. We have
\begin{proposition}
$P(t)\ls 1.$
\end{proposition}

\begin{proof}
For notational simplicity, we define 
$$
g(s) \eqdef 1+A(s)^2+\left(\int_0^s A(s')^8 ds'\right)^{\frac 12}.
$$
Notice that by the assumption of Theorem \ref{main.theorem}, $g$ is integrable in time.

By \eqref{char}, the momentum support $P$ can be estimated by
$$
P(t)\ls 1+ \int_0^t (|E(s)|+|B(s)|) ds.
$$
By the estimates from Propositions \ref{field.est.1}, \ref{field.est.2} and \ref{field.est.3} we have 
\beaa
P(t) \ls 1+\log P(t)\int_0^t \left(g(s)+\int_0^s A(s')^4 ds'\right)P(s) ds,
\eeaa
where the implicit constant depends on $t$.
Notice that 
$$
\int_0^s A(s')^4 ds'\ls  s^{\frac 12}\left(\int_0^s A(s')^8 ds'\right)^{\frac 12}.
$$
Thus, we have
\bea \notag 
P(t) \leq C_0(t)+C_0(t)\log P(t) \int_0^t g(s) ~ P(s) ds,
\eea
where $C_0(t)$ is a positive continuous function on $\mathbb R$. Let $h(t)=\frac{P(t)}{\log P(t)}$. Then, since $P(s) \ls h(s)\log h(s)$, the previous estimate implies
\bea \label{main.est.h}
h(t) \leq C_1(t)+C_1(t) \int_0^t g(s) ~ h(s)\log h(s)~ ds,
\eea
for some positive continuous function $C_1(t)$. 
This is the main estimate that we will need.

We want to show that for fixed $T>0$,
\bea
h(t)\leq C_T \quad\forall t\in [0,T).\label{goal}
\eea
For any fixed $T$, we can assume without loss of generality that $C_1(t)$ is a positive constant.

To achieve \eqref{goal}, we use a continuity argument. Assume as a bootstrap assumption that
\bea
h(t)\leq 2C_1\exp\left(\exp\left(\Delta_0\int_0^t g(s') ds'\right)\right)\label{bootstrap}
\eea
for all $t<T$, where $\Delta_0$ is a large constant to be chosen later.

Then, by \eqref{main.est.h}, we have
\begin{multline}\notag
h(t)
\leq C_1+C_1 \int_0^t g(s)   h(s)\log h(s) ds
\\
\leq C_1+ 2C_1^2 \int_0^t g(s) e^{e^{\Delta_0\int_0^s g(s') ds'}}
\left(\log(2C_1)+e^{\Delta_0\int_0^s g(s') ds'} 
\right) ds
\\
\leq C_1+ 2C_1^2 \int_0^t g(s) e^{e^{\Delta_0\int_0^s g(s') ds'}}
(\log(2C_1)+1)\exp \left(\Delta_0\int_0^s g(s') ds'\right) ds
\\
\leq 
C_1+ 2C_1^2(\log(2C_1)+1) \Delta_0^{-1} \int_0^{e^{\int_0^t g(s') ds'}} e^{e^{\Delta_0\int_0^s g(s') ds'}} d(e^{\Delta_0\int_0^s  g(s') ds'})
\\
=C_1+2C_1^2(\log(2C_1)+1) \Delta_0^{-1}\left(e^{e^{\Delta_0\int_0^t g(s') ds'}} -e\right)\\
\leq 
\left(C_1+2C_1^2(\log(2C_1)+1) \Delta_0^{-1}\right)e^{e^{\Delta_0\int_0^t g(s') ds'}},
\end{multline}
where in the last steps we have used $e^{e^{\Delta_0\int_0^t g(s') ds'}}\geq 1$.   By choosing $\Delta_0$ sufficiently large depending on $C_1$, we have
$$C_1+2C_1^2(\log(2C_1)+1) \Delta_0^{-1} \leq \frac 32 C_1.$$
Thus we have closed the bootstrap assumption \eqref{bootstrap}. Therefore, \eqref{goal} holds. Returning to the definition $h(t)=\frac{P(t)}{\log P(t)}$, we have thus showed that $P(t)\ls 1$, as desired.
\end{proof}


\end{document}